\documentclass[twoside]{article}
\usepackage{graphicx}
\usepackage{ae,aecompl}
\usepackage{amssymb}
\usepackage{amsmath}
\usepackage[all]{xy}
\usepackage{mor}
\received{}                              %
\revised{}                               %
\pubyear{0x}                             %
\pubmonth{Xxxxxxx}                       %
\volume{xx}                              %
\issue{x}                                %
\pages{xxx--xxx}                         %
\firstpage{0xxx}                         %
\DOI{10.1287/moor.xxxx.xxxx}             %
\startpagenumber{1}                      %

\newcommand{\argmin}{\mathop{\mathrm{arg\,min}}}

\newcommand{\Knn}{{\mathcal K}_{nn}}
\newcommand{\pinn}{\pi^{*}_{\Knn}}
\newcommand{\beps}{\varepsilon}
\newcommand\pig[2]{\pi^{#2}_{G}(#1)}
\newcommand\pin[2]{\pi^{#2}_{\Knn}(#1)}
\newcommand\pit[2]{{\pi^{#2}_{\cal T}(#1)}}
\newcommand\mg[3]{\langle #1\to#2\rangle^{#3}_{G}}
\newcommand\mgn[3]{\langle #1\to#2\rangle^{#3}_{\Knn}}
\newcommand\mgt[3]{\langle #1\to#2\rangle^{#3}_{\cal T}}
\newcommand\ew[2]{\big\|#1,#2\big\|_G}
\newcommand\ewn[2]{\big\|#1,#2\big\|_{\Knn}}
\newcommand\ewt[2]{\big\|#1,#2\big\|_{\cal T}}
\newcommand\dif[3]{\ew{#1}{#2}-\mg{#1}{#2}{#3}}
\newcommand\difn[3]{\ewn{#1}{#2}-\mgn{#1}{#2}{#3}}
\newcommand\dift[3]{\ewt{#1}{#2}-\mgt{#1}{#2}{#3}}
\renewcommand{\updatemarkboth}{}

\title{Belief propagation: an asymptotically optimal algorithm \\
for the random 
assignment problem}
\ShortTitle{BP for random assignment}
\ShortAuthors{Salez and Shah}
\NumberOfAuthors{2}
\FirstAuthor{Justin Salez}
\FirstAuthorAddress{INRIA and \'Ecole Normale Supérieure de Paris}
\FirstAuthorEmail{justin.salez@ens.fr}
\SecondAuthor{Devavrat Shah}
\SecondAuthorAddress{EECS, Massachusetts Institute of Technology}
\SecondAuthorEmail{devavrat@mit.edu}


\MSCcodes{Primary:    60C05, 68T20, 68W40; Secondary:  05C80, 82B44.    } 

\ORMScodes{Primary: Network/graphs matchings,  Probability stochastic model; Secondary:  Mathematics combinatorics and fixed points.} 

\begin{document}
\maketitle

\begin{abstract}

The random assignment problem asks for the minimum-cost perfect  
 matching in the complete $n\times n$ bipartite graph $\Knn$ with i.i.d. edge 
weights, say uniform on $[0,1]$. In a remarkable work by Aldous
(2001), the optimal cost was
 shown to converge to $\zeta(2)$ as $n\to\infty$, as
 conjectured by M\'ezard and Parisi (1987) through the so-called cavity
 method.
 The latter also suggested a non-rigorous decentralized  strategy for finding the optimum, which turned out to be an instance of the Belief Propagation (BP) heuristic discussed by Pearl (1987). In this paper
 we use the objective method to analyze the
 performance of BP as the size of the underlying graph becomes large. Specifically, we establish that the dynamic of BP on $\Knn$
 converges in distribution as $n\to\infty$ to an appropriately defined
 dynamic on the Poisson Weighted Infinite Tree, and we then prove
 correlation decay for this limiting dynamic. 
As a consequence, we obtain that  BP finds an
asymptotically correct assignment in $O(n^2)$ time only. This contrasts with both the worst-case upper bound
for convergence of BP derived by Bayati, Shah and Sharma (2005) and
the best-known computational cost of $\Theta(n^3)$ achieved by Edmonds and Karp's algorithm (1972).
\end{abstract}
\section{Introduction.}
Given a matrix of $n^2$ costs $(X_{i,j})_{1\leq i,j\leq n}$,
the assignment problem consists of determining a permutation $\pi$ of
$\{1,\ldots,n\}$ whose total cost $\sum_{i=1}^{n}{X_{i,\pi(i)}}$ is minimal. This is equivalent to finding a minimum-weight complete matching
in the $n\times n$ complete bipartite graph whose $n^2$ edges are
weighted by the $(X_{i,j})$. Recall that a complete matching on
a graph is a subset of pairwise disjoint
edges covering all vertices. Here we consider
the so-called random assignment problem where the $(X_{i,j})$ are
i.i.d. with cumulative distribution function denoted by $H$, i.e. $H(t)=\mathbb
P(X_{i,j}\leq t)$. We let $\Knn$ denote the resulting randomly weighted
$n\times n$ bipartite graph and $\pinn$ its optimal matching. Observe that the continuity of $H$ is a necessary and sufficient condition for
$\pinn$ to be a.s. unique. We are interested in the
 convergence of the BP heuristic for finding $\pinn$ as $n$ increases to infinity.

\subsection{Related Work.}
Although it seems cunningly simple, the assignment problem has led to
rich development in combinatorial probability and algorithm design since
the early 1960s. Partly motivated to obtain insights for better algorithm
design,  the question  of finding asymptotics of the average cost of
$\pinn$  became of
great interest (see \cite{walkup, DFM, Karp,Lazarus, Olin, GK, CS}).
In 1987, through cavity method based
calculations, M\'ezard and Parisi \cite{MezPar} conjectured that, for
Exponential(1) edge weights,
$$\mathbb E\left[\sum_{i=1}^{n}{X_{i,\pinn(i)}}\right]\xrightarrow[n\to\infty]{}\zeta(2).$$
This was rigorously established by Aldous \cite{zeta} more than a decade
later, leading to the formalism of ``the objective method'' (see
survey by Aldous and Steele \cite{objective}). In 2003, an exact version of the
above conjecture 
was independently established  by Nair, Prabhakar and Sharma \cite{NPS}
and Linusson and W\.astlund \cite{LW}.

On the algorithmic aspect, the assignment problem has been extremely well
studied and its consideration laid foundations for the rich theory of
network flow algorithms. The best known algorithm is by Edmonds and Karp
\cite{edmund-karp} and takes $O(n^3)$ operations in the worst-case for arbitrary
instance. For i.i.d. random edge weights, Karp \cite{mnlogn}
designed a special implementation of the augmenting path approach using priority
queues  that works in expected time $O(n^2 \log n)$. Concurrently, the statistical
physics-based approach mentioned above suggested a non-rigorous
decentralized strategy which turned out to be an instance of the more general BP heuristic, popular in artificial intelligence (see, book by Pearl
\cite{pearl} and work by Yedidia, Freeman and Weiss \cite{BP}). In a recent
work, one of the authors of the present paper, Shah along with Bayati and
Sharma \cite{BSS}, established correctness of this iterative scheme for any
instance of the assignment problem, as long as the optimal solution is
unique. More precisely, they showed exact convergence within at most
$\lceil \,\frac{2 n \max_{i,j}  X_{i,j}}{\varepsilon}\rceil$ iterations,
where $\varepsilon$ denotes the difference of weights between optimum and
second optimum. This upper bound is always greater than $n$, and can be shown to
scale like $\Theta(n^2)$ as $n$ goes to infinity in the random model. Since
each iteration of the BP algorithm needs $\Theta(n^2)$ operations to be
performed, one is left with an upper bound of $O(n^4)$ for the total computation
cost. However, simulation studies tend to show much better performances on
average than what is suggested by this worst-case analysis.

\subsection{Our contribution.}
Motivated by the above discussion, we consider here the question of
determining the convergence rate of BP for the random assignment
problem. We establish that, for a large class of edge-weight distributions,
the number of iterations required in order to find an almost optimal
assignment remains in fact bounded as $n\to\infty$. Thus, the total
computation cost scales as $O(n^2)$ only, in sharp contrast with both the
worst-case upper bound for exact convergence of BP derived in \cite{BSS}
and the $\Theta(n^3)$ bound achieved by Edmonds and Karp's
algorithm. Clearly, no algorithm can perform better than $\Omega(n^2)$,
since it is the size of the input. That is, BP is an asymptotically optimal
algorithm on average.

\section{Result and organization.}

\subsection{BP algorithm.}
As we shall see later, the dynamics of BP on $\Knn$ happens to converge to
the dynamics of BP on a limiting infinite tree. Therefore, we define the BP
algorithm for an arbitrary weighted graph $G=(V,E)$. We use notation that
the weight of $\{v,w\} \in E$ is $\|v,w\|_G$. By $w \sim v$, we denote that
$w$ is a neighbor of $v$ in $G$. Note that a complete matching on $G$ can
be equivalently seen as an involutive mapping $\pi_G$ connecting each
vertex $v$ to one of its neighbors $\pi_G(v)$. We shall henceforwards use
this mapping representation rather than the edge set description.

The BP algorithm is distributed and iterative. Specifically, in each
iteration $k\geq 0$, every vertex $v\in V$
sends a real-valued message $\mg{v}{w}{k}$ to each of its neighbor $w\sim
v$ as follows:
\begin{itemize}
\item {\em initialization rule: }
\begin{equation}
\label{bp0}
\mg{v}{w}{0}  = 0 \frac{}{};
\end{equation}
\item {\em update rule: }
\begin{equation}
\mg{v}{w}{k+1} = \min_{u\sim v,u\neq w}{\left\{\dif{u}{v}{k}\right\}}.
\end{equation}
\end{itemize}
Based on those messages, every vertex $v\in V$ estimates the neighbor
$\pi^k_G(v)$ to which it connects as follows:
\begin{itemize}
\item {\em decision rule: }
\begin{equation}
\label{beliefalgo} \pi^k_{G}(v)=\argmin_{u\sim
v}{\left\{\dif{u}{v}{k}\right\}}.
\end{equation}
\end{itemize}
When $G=\Knn$, \cite{BSS} ensures convergence of $\pi^k_{\Knn}$ to the
optimum $\pi^*_{\Knn}$ as long as the latter is unique, which holds almost
surely if and only if $H$ is continuous.  The present paper asks about the
typical rate of such a convergence, and more precisely its dependency upon
$n$ as $n$ increases to $\infty$.

\subsection{Result.}
In order to state our main result, we introduce the normalized Hamming
distance between two given assignments $\pi,\pi'$ on a graph $G=(V,E)$ :
\begin{eqnarray*}
d(\pi,\pi')=\frac{1}{|V|}\textrm{ card}\Big\{v\in V, \pi(v)\neq \pi'(v)\Big\}.
\end{eqnarray*}
\begin{theorem}
 \label{tm:unifL1conv}
Assume the cumulative distribution function $H$ satisfies:
\begin{enumerate}
\item[A1.] Regularity : $H$ is continuous and $H'(0^+)$ exists and is non-zero;
\item[A2.] Light-tail property : as $t\to\infty$, $H(t)=1-O\left(e^{-\beta
t}\right)$\! for some $\beta>0$.
\end{enumerate}
Then,
\begin{equation*}
\label{eq:unifL1conv}
\limsup_{n\to\infty}{\mathbb
  E\Big[d\big(\pi^k_{{\cal K}_{nn}},\pi^*_{{\cal
K}_{nn}}\big)\Big]}\xrightarrow[k\to\infty]{}0.
\end{equation*}
\end{theorem}
In other words, given any $\varepsilon > 0$, there exists
$k(\beps),n(\beps)$ such that the expected fraction of non-optimal
row-to-column assignments after $k(\beps)$ iterations of the BP algorithm
on a random $n\times n$ cost array is less than $\varepsilon$, no matter
how large $n\geq n(\beps)$ is. Consequently, the probability to get more
than any given fraction of errors can be made as small as desired within
finitely many iterations, independently of $n$. Since each iteration
requires $O(n^2)$ operations, the overall computation cost scales as
$O(n^2)$ only, with constant depending on the admissible error. This
applies for a wide class of cost distributions, including uniform over
$[0,1]$ or Exponential.

\begin{remark}
  It may be the case that the $\varepsilon$ fraction of wrong row-to-column
  assignments results in local violations of the matching
  property. Depending on the context of application, this might be quite
  unsatisfactory. However, such an ``$\varepsilon-$feasible matching'' can
  easily be modified in order to produce an honest matching without
  substantially increasing the total cost (see \cite[Proposition 2]{asymp}
  for details).
\end{remark}

\subsection{Organization.}

The remaining of the paper is dedicated to proving Theorem
\ref{tm:unifL1conv}. Although it is far from being an implication of the
result by Aldous \cite{zeta}, it utilizes the machinery of local
convergence, and in particular the Poisson Weighted Infinite Tree $\cal T$
appearing as the limit of $({\cal K}_{nn})_{n\geq 1}$.  These notions are
recalled in Section \ref{sec:prelim}. The diagram below illustrates the
three steps of our proof : Theorem \ref{tm:unifL1conv} corresponds to
establishing the top-horizontal arrow, which is done by establishing the
three others.
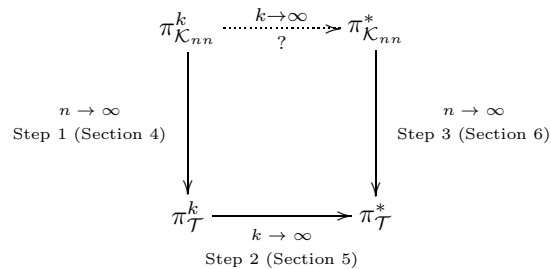
\begin{figure}[h!]
$$\xymatrix @!0 @R=2.5cm @C=2.5cm{
   \pi^k_{{\cal K}_{nn}} \ar[d]_{\tiny{\begin{array}{c} n\to\infty\\
         \textrm{Step 1 } (\textrm{Section }\ref{sec:step1})\end{array}}}  \ar@{.>}[r]^{k\to\infty}_{?}
   & \pi^*_{{\cal K}_{nn}} \ar[d]^{\tiny{\begin{array}{c} n\to\infty\\
         \textrm{Step 3 } (\textrm{Section }\ref{sec:step3})\end{array}}} \\
   \pi^k_{{\cal T}} \ar[r]_{\tiny{\begin{array}{c} k\to\infty\\
         \textrm{Step 2 }(\textrm{Section }\ref{sec:step2})\end{array}}} & \pi^*_{\cal T}
 }$$
\label{fig:proof}
\end{figure}
\vspace{-.1in}
\begin{itemize}
\item[1.] First (Section \ref{sec:step1}), we prove that BP's behavior on
  ${\cal K}_{nn}$  ``converges'' as $n\to\infty$ to its behavior on $\cal
  T$. This is formally stated as Theorem \ref{tm:continuity} and corresponds to the left vertical arrow above.
\item[2.]  Second (Section \ref{sec:step2}), we
establish convergence of BP on $\cal T$. This is summarized as Theorem
\ref{tm:strongconv} and corresponds to the bottom horizontal arrow in the
above diagram. We note that Theorem \ref{tm:weakconv}
resolves an open problem stated by Aldous and Bandyopadhyay (\cite[Open Problem \# 62]{maxtype}).
\item[3.] Third (Section
  \ref{sec:step3}), the connection between the fixed point on $\cal T$  and
  the optimal matching on ${\cal K}_{nn}$ is provided by the work by Aldous \cite{zeta} -- corresponding to the vertical right arrow and
  stated as Theorem \ref{tm:aldous2}. We use it to complete our proof.
\end{itemize}

\section{Preliminaries.}
\label{sec:prelim}

We recall here the necessary framework introduced by Aldous in \cite{zeta}.
Consider a rooted, edge-weighted and connected graph $G$, with distance
between two vertices being defined as the infimum over all paths connecting
them of the sum of edge weights along that path. For any $\varrho >0$,
define the $\varrho-${restriction} of $G$ as the subgraph $\lceil
G\rceil_{\varrho}$ induced by the vertices lying within distance $\varrho$
from the root. Call $G$ a {\em geometric graph} if $\lceil
G\rceil_{\varrho}$ is finite for every $\varrho > 0$.

\begin{definition}[local convergence]
\label{df:localconv}
Let $G, G_1,G_2,\ldots$ be geometric graphs. We say that $(G_n)_{n \geq 1}$
converges to $G$ if for every $\varrho>0$ such that no vertex in $G$ is at
distance exactly $\varrho$ from the root the following holds:

\begin{enumerate}
\item[1.] $\exists n_\varrho\in\mathbb N$ s.t. the
  $\displaystyle{\lceil{G_n}\rceil_\varrho, n\geq n_\varrho}$ are all
  isomorphic\footnote{An isomorphism from $G=(V,\varnothing,E)$ to
    $G'=(V',\varnothing',E')$, denoted $\gamma \colon G \rightleftharpoons
    G'$, is simply a bijection from $V$ to $V'$ preserving the root
    ($\gamma\big(\varnothing\big)=\varnothing'$) and the structure
    ($\forall (x,y)\in V, \{\gamma(x),\gamma(y)\} \in E' \Leftrightarrow
    \{x,y\} \in E$).} to $\lceil G \rceil_\varrho$ ;
\item[2.] The corresponding isomorphisms $\gamma^\varrho_n\colon \lceil
  G\rceil_{\varrho}\rightleftharpoons \lceil G_n\rceil_{\varrho},{n\geq
    n_\varrho}$ can be chosen so that for every edge $\{v,w\}$ in $\lceil
  G\rceil_{\varrho}$:
\begin{equation*}
  \big\|\gamma^\varrho_n(v),\gamma^\varrho_n(w)\big\|_{G_n}\xrightarrow[n\to\infty]{}\ew{v}{w}.
\end{equation*}
\end{enumerate}
In the case of labeled geometric graphs, each oriented edge $(v,w)$ is also
assigned a label $\lambda(v,w)$ taking values in some Polish space
$\Lambda$. Then the isomorphisms $(\gamma^\varrho_n)_{n\geq n_\varrho}$
have to moreover satisfy the following:
\begin{enumerate}
\item[3.]  For every oriented edge $(v,w)$ in $\lceil G \rceil_\varrho$,
  $\lambda_{G_n}\left(\gamma^\varrho_n(v),\gamma^\varrho_n(w)\right)\xrightarrow[n\to\infty]{}\lambda_G\left(v,w\right).$
\end{enumerate}
\end{definition}
The intuition behind this definition is the following: in any arbitrarily
large but fixed neighborhood of the root, $G_n$ should look very much like
$G$ for large $n$, in terms of structure (part 1), edge weights (part 2)
and labels (part 3).  With little work, one can define a distance that
metrizes this notion of convergence and makes the space of (labeled)
geometric graphs complete and separable. As a consequence, one can import
the usual machinery related to the theory of weak convergence of
probability measures. We refer the reader unfamiliar with these notions to
the excellent book of Billingsley \cite{Bil99}.

Now, consider our randomly weighted $n\times n$ bipartite graph $\Knn$ as a
random geometric graph by fixing an arbitrary root, independently of the
edge weights. Then the sequence $(\Knn)_{n\geq 1}$ happens to converge
locally in distribution to an appropriately weighted infinite random
tree. Before we formally state this result known as the ``PWIT Limit
Theorem'' \cite{zeta, maxtype}, we introduce some notations that will be
useful throughout the paper. We let $\cal V$ denote the set of all finite
words over the alphabet ${\mathbb N}^*$, $\varnothing$ the empty word,
``$\cdot$'' the concatenation operation and for any $v\in {\cal V}^*={\cal
  V}\setminus \{\varnothing\}$, $\dot{v}$ the word obtained from $v$ by
deleting the last letter. We also set ${\cal E}=\left\{\{v,v.i\},v\in{\cal
    V},i\geq 1\right\}$. The graph ${\cal T}=({\cal V}, {\cal E})$ thus
denotes an infinite tree with $\varnothing$ as root, letters as the nodes
at depth $1$, words of length $2$ as the nodes at depth $2$, etc. Now,
consider a collection $\left(\xi^v=\xi^v_1,\xi^v_2\ldots\right)_{v\in{\cal
    V}}$ of independent, ordered Poisson point processes with intensity $1$
on ${\mathbb R}^+$, and assign to edge $\{v,v.i\}\in{\cal E}$ the weight
$\ewt{v}{v.i}=\xi^{v}_i$. This defines the law of a random geometric graph
$\cal T$ called the ``Poisson Weighted Infinite Tree'' (PWIT).
\begin{theorem}[Pwit Limit Theorem, Aldous \cite{asymp,zeta}]
\label{tm:PWITlimit}  Under assumption $A1$ on $H$:
\begin{equation}
\label{eq:PWITlimit} nH'(0^+) {\cal
K}_{nn}\xrightarrow[n\to\infty]{{\cal D}}\cal T,
\end{equation}
in the sense of local weak convergence of geometric graphs.
\end{theorem}
\begin{remark}
  To get rid of scaling factors, we will henceforth multiply all edge
  weights in ${\cal K}_{nn}$ by $nH'(0^+)$. Observe that both the optimal
  matching $\pi^*_{{\cal K}_{nn}}$ and BP estimates $\pi^k_{{\cal K}_{nn}},
  {k\geq 0}$ remain unaffected.
\end{remark}
\section{First step: convergence to a limiting dynamic as $n\to\infty$.}\label{sec:step1}
In this section we deduce from the PWIT Limit Theorem that the behavior of
BP when running on ${\cal K}_{nn}$ ``converges'' as $n\to\infty$ to its
behavior when running on $\cal T$. To turn this idea into a rigorous
statement, let us encode the execution of BP as labels attached to the
oriented edges of the graph. Specifically, given a geometric graph $G$ and
an integer $k\geq 0$, we define the $k^{th}-$step configuration of BP on
$G$, denoted by $(G,\mg{\cdot}{\cdot}{k},\pi^k_G )$, as the labeled
geometric graph obtained by setting the label of any oriented edge $(v,w)$
in $G$ to be the couple $(\mg{v}{w}{k},\mathbf{1}_{\{w=\pig{v}{k}\}})$. We
can now state and prove the main theorem of the present section.
\begin{theorem}[Continuity of BP]
\label{tm:continuity} Consider an almost sure realization of the PWIT limit
Theorem: 
\begin{equation}
\label{eq:pwitlimitas}
\Knn\xrightarrow[n\to\infty]{a.s.}{\cal T}.
\end{equation}
Then for every fixed $k\geq 0$, the $k^{th}-$step configuration of BP on
$\Knn$ converges locally in probability to the $k^{th}-$step configuration
of BP on $\cal T$ :
\begin{eqnarray}
\label{eq:continuity1}
\left(\Knn,\mgn{\cdot}{\cdot}{k},\pi^k_{{\cal
      K}_{nn}}\right)
&\xrightarrow[n\to\infty]{{P}}&
\left({\cal T},\mgt{\cdot}{\cdot}{k},\pi^k_{\cal T}\right).
\end{eqnarray}
\end{theorem}
\begin{proof}
  Let us (redundantly) re-label the vertices of $\Knn$ by words of $\cal V$
  in a manner that yields to consistent comparison between the messages on
  $\Knn$ and those on $\cal T$. To begin with, let the empty word
  $\varnothing$ represent the root of $\Knn$ and words $1,2,\cdots,n$ its
  immediate neighbors, ordered by increasing weight of the edge connecting
  them to the root. Then, inductively, if word $v\in \cal V^*$ represents
  some vertex $x\in \Knn$ and $\dot{v}$ some $y\in \Knn$, then let the
  words $v.1,v.2,\cdots,v.(n-1)$ represent the $n-1$ neighbors of $x$
  distinct from $y$ in $\Knn$, again ordered by increasing weight of the
  corresponding edge. Note that this definition makes almost surely sense
  since the edge weights are pairwise distinct (by continuity of $H$). In
  fact, it follows from an easy induction on $v\in\cal V$ that the vertex
  represented by $v$ in $\Knn$ is nothing but $\gamma_n^\varrho(v)$ as soon
  as $\varrho$ and $n$ are large enough, where
  $\gamma_n^\varrho\colon\lceil{\cal T}\rceil_\varrho\rightleftharpoons
  \lceil\Knn\rceil_\varrho$ is the (random) isomorphism involved in the
  definition of the local convergence (\ref{eq:pwitlimitas}). In
  particular,
\begin{equation}
\forall \{v,w\}\in{\cal E},\,
\ewn{v}{w}\xrightarrow[n\to\infty]{a.s.}\ewt{v}{w}.
\end{equation}
With this relabeling in hand, the desired convergence
(\ref{eq:continuity1}) can now be written:
\begin{equation}
\label{eq:continuity2} \forall \{v,w\}\in {\cal E},
\mgn{v}{w}{k}\xrightarrow[n \to \infty]{P}\mgt{v}{w}{k}\textrm{ and }
\forall v \in {\cal V}, \pin{v}{k}\xrightarrow[n \to \infty]{P}\pit{v}{k}
.
\end{equation}
The recursive nature of the messages almost compels one to think of proving
(\ref{eq:continuity2}) by induction over $k$. The base case of $k=0$ is
trivial. However, when trying to go from step $k$ to step $k+1$ one soon
gets confronted by a major hinder: the update and decision rules
(\ref{bp0}) and (\ref{beliefalgo}) are not continuous with respect to local
convergence. Indeed, writing:
\begin{eqnarray*}
\label{eq:inf1} 
\mgn{v}{w}{k+1}&=&\min_{\tiny{\begin{array}{c}u\in\{v.1,\ldots,v.(n-1),\dot{v}\}\\u\neq
  w\end{array}}}\Big\{\difn{u}{v}{k}\Big\}\\
\textrm{and }\pin{v}{k}&=&\argmin_{u\in\{v.1,\ldots,v.(n-1),\dot{v}\}}\Big\{\difn{u}{v}{k}\Big\},
\end{eqnarray*}
one can not simply invoke convergence of each term inside the $\min$ and
$\argmin$ to conclude, because there are unboundedly many such terms as
$n\to \infty$. Remarkably enough, it turns out that under assumption A2, we
can in fact restrict ourselves to a uniformly bounded number of them with
probability as high as desired, as stated in the following lemma.
\end{proof}

\begin{lemma}[Uniform control on essential messages]
\label{lm:unifctrl}
For all $v\in{\cal V}$ and $k\geq 0$ :
\begin{equation*}
\limsup_{n\to\infty}{\mathbb
P\left(\argmin_{1\leq i< n}{\Big\{\difn{v}{v.i}{k}\Big\}} \geq
i_0\right)} \xrightarrow[i_0\to\infty]{}0.
\end{equation*}
\end{lemma}
The proof of this Lemma is long and technical and hence is defered to Appendix \ref{app1}
\section{Second step: analysis of BP on PWIT.}\label{sec:step2}
In light of Theorem \ref{tm:continuity}, one can replace the asymptotic
analysis of BP on $\Knn$ as $n$ becomes large by the direct study of BP's
dynamics on the limiting PWIT. Formally, we are interested in the limiting
behavior of the random process defined for all $v\in {\cal V}^*$ by the
recursion:
\begin{equation} 
\label{eq:recas}\mgt{v}{\dot{v}}{k+1} = \min_{i\geq
1}{\left\{\dift{v}{v.i}{k}\right\}},
\end{equation}
where the initial values $\left(\langle v\to\dot{v}\rangle^0_{\cal
    T}\right)_{v\in{\cal V}^*}$ are i.i.d. random variables independent of
$\cal T$ ($0$ in the case of our algorithm). The fact that the above $\min$
is a.s. well defined despite the infinite number of terms will become clear
later (see Lemma \ref{lm:welldefiniteness}). For the time being, it is
sufficient to consider it as a $\overline{\mathbb R}$-valued infimum. First
observe that at any given time $k$ all $\langle v\to\dot{v}\rangle^k_{\cal
  T}, v\in {\cal V}^*$ share the same distribution, owing to the natural
spatial invariance of the PWIT.  Moreover, if $F$ denotes the corresponding
tail distribution function at a given time, a straightforward computation
(see for instance \cite{zeta}) shows that the tail distribution function
$TF$ obtained after a single application of update rule (\ref{eq:recas}) is
given by:
\begin{equation*}
\label{eq:recweak} TF:x\mapsto
\exp{\left(-\int_{-x}^{+\infty}{F(t)\,dt}\right)}.
\end{equation*}
This defines an operator $T$ on the space $\cal D$ of tail distribution
functions of $\overline{\mathbb R}-$valued random variables, i.e.
non-increasing corlol\footnote{\textbf{c}ontinuous \textbf{o}n the
  \textbf{r}ight, \textbf{l}imit \textbf{o}n the \textbf{l}eft} functions
$F \colon {\mathbb R} \to [0,1]$. $T$ is known to have a unique fixed point
(see \cite{zeta}), the so-called \textit{logistic} distribution:
\begin{equation*}
F^*:x \mapsto \frac{1}{1+e^x}.
\end{equation*}
Our first step will naturally consist in studying the dynamics of $T$ on
$\cal D$.
\subsection{Weak attractiveness.}

Finding the domain of attraction of $F^*$ under operator $T$ is not known
and has been listed as an open problem by Aldous and Bandyopadhyay
(\cite[Open Problem \# 62]{maxtype}). In what follows, we answer this
question and more. We fully characterize the asymptotic behavior of the
successive iterates $(T^kF)_{k\geq 0}$ for any initial distribution
$F\in\cal D$.

First observe that $T$ is anti-monotone with respect to pointwise order:
\begin{equation}
\label{decr} F_1\leq F_2 \Longrightarrow TF_1 \geq TF_2.
\end{equation}
This suggests considering the non-decreasing second iterate $T^2$. Unlike
$T$, $T^2$ admits infinitely many fixed points. To see this, let $\theta_t$
($t\in\mathbb R$) be the $t-$shift operator defined on $\cal D$ by
$\theta_tF\colon x\mapsto F(x-t)$.  Then a trivial change of variable
gives:
\begin{equation}
\label{tshift} T\circ\theta_t = \theta_{-t}\circ T.
\end{equation}
Therefore, it follows that $T^2(\theta_tF^*)=\theta_t(T^2F^*)=\theta_tF^*$ for all
$t\in\mathbb R$. That is, the $\theta_t F^*, t\in\mathbb R$ are fixed points of $T^2$. These considerations lead us to introduce the key tool of our analysis:
\begin{definition}
For $F\in {\cal D}$, define the transform $\widehat F$ as follows :
$$\forall x\in\mathbb R, \widehat F(x) = x+\ln\left(\frac{F(x)}{1-F(x)}\right).$$
\end{definition}
Intuitively, ${\widehat F}$ represents the local shift (along the X-axis)
between $F$ and $F^*$.  Indeed, it enables us to express any $F\in\cal D$
as a locally deformed version of $F^*$ via the following straightforward
inversion formula: $$\forall x, F(x)=F^*(x-{\widehat
  F}(x))=\theta_{{\widehat F}(x)}F^*(x).$$ In particular,
$\theta_{t_1}F^*\leq F\leq \theta_{t_2}F^*\Longleftrightarrow t_1\leq
\widehat{F}\leq t_2$, and $F=\theta_tF^*$ if and only if $\widehat{F}$ is
constant on $\mathbb R$ with value $t$. In that sense, the maximal
amplitude of the variations of $\widehat F$ on $ \mathbb R$ tells something
about the distance between $F$ and the family of fixed points
$\{\theta_tF^*,{t\in\mathbb R}\}$. Thus, the action of $T$ on those
variations appears to be of crucial importance and will be at the center of
our attention for the rest of this sub-section. We now state three lemmas
whose proofs are given in Appendix \ref{app2}.
\begin{lemma}
\label{lm:boundedness} Let $F\in{\cal D}\setminus\{0\}$ such that $\displaystyle{\int_{0}^{+\infty}F<\infty}$. Then, $\widehat {T^4F}$ is bounded on $\mathbb R$.
\end{lemma}
\begin{lemma}
\label{lm:maxamplitude} If $F\in \cal D$ is such that $\widehat{F}$ is
bounded, then $\widehat{TF}$ is bounded too, and moreover:
$$ -\sup_{\mathbb
R}{\widehat{F}}\leq \inf_{\mathbb R}{\widehat{TF}}  \leq \sup_{\mathbb
R}{\widehat{TF}}  \leq  -\inf_{\mathbb R}{\widehat{F}}.$$
Further, if $\widehat{F}$ is not constant then this contraction becomes strict under a second iteration : 
$$ \inf_{\mathbb R}{\widehat{F}}<\inf_{\mathbb R}{\widehat{T^2F}}\leq\sup_{\mathbb R}{\widehat{T^2F}}  <  \sup_{\mathbb
R}{\widehat{F}}.$$

\end{lemma}
\begin{lemma}
  \label{lemunifint} Let $F\in\cal D$ be such that $\widehat F$ is
  bounded. Then, $\widehat {T^kF}$ is continuously differentiable for
  $k\geq 2$, and the family of derivatives $(\widehat{T^kF})',k\geq 3$ is
  uniformly integrable:
\begin{equation*}
\sup_{k\geq 3}{\int_{|x|> M}\left|(T^kF)'(x)\right|dx}\xrightarrow[M
\to\infty]{}0.
\end{equation*}
\end{lemma}
We are now in position to provide a complete description of the dynamics
of $T$ on $\cal D$.
\begin{theorem}[Dynamics of $T$ on $\cal D$]
\label{tm:weakconv} 
Let $F\in\cal D$.
Assume $F$ is not the $0$ function and $\int_{0}^{\infty}F < +\infty$
(otherwise $(T^kF)_{k\geq 1}$ trivially alternates between the $0$ and $1$
functions). Then, there exists a constant $\gamma\in\mathbb R$ dependent on
$F$ such that $\widehat{T^{2k}F}\xrightarrow[k\to\infty]{}\gamma\textrm{
  and }\widehat{T^{2k+1}F}\xrightarrow[k\to\infty]{}-\gamma,\textrm{
  uniformly on }\mathbb R.$
In particular, 
\begin{equation*}
T^{2k}F \xrightarrow[k\to\infty]{} \theta_{\gamma}F^*\textrm{ and }T^{2k+1}F \xrightarrow[k\to\infty]{} \theta_{-\gamma}F^*,\textrm{ uniformly on }\mathbb R.
\end{equation*}
\end{theorem}

\begin{proof}
By Lemma \ref{lm:boundedness}, one can choose a
large enough $M\geq 0$ for $T^4F$ to lie in the subspace
$${\cal D}_M=\{F\in {\cal D}, 
-M\leq \widehat F\leq M\}=\{F\in {\cal D}, \,\theta_{-M}F^*\leq F\leq
\theta_{M}F^*\}.$$ Lemma \ref{lm:maxamplitude} guarantees the stability of
${\cal D}_M$ under the action of $T$, so the whole sequence $(T^{k}F)_{k\geq 4}$ remains in ${\cal D}_M$. Even better, the
bounded real sequences $(\inf_{\mathbb R}\widehat{T^{2k}F})_{k\geq 2}$ and
$(\sup_{\mathbb R}\widehat{T^{2k}F})_{k\geq 2}$ are monotone, hence
convergent, say to $\gamma^-$ and $\gamma^+$ respectively. All we have to
show is that $\gamma^-=\gamma^+$; convergence of
$(\widehat{T^{2k+1}F})_{k\geq 2}$ to the opposite constant will then simply
follow from property (\ref{tshift}) .

By Arzela-Ascoli theorem, the family of (clearly bounded and 1-Lipschitz)
functions $(T^{2k}F)_{k\geq 2}$ is relatively compact with respect to
compact convergence. Thus, there exists a convergent sub-sequence:
\begin{equation}
\label{eq:conv} T^{2\varphi(k)}F\xrightarrow[k\to\infty]{}F_\infty.
\end{equation}
From the uniform continuity of $y\mapsto \ln\frac{y}{1-y}$ on every compact
subset of $]0,1[$ (Heine's theorem), it follows that the restriction of the
$\widehat{\cdot}$ transform to ${\cal D}_M$ is continuous with respect to
compact convergence. Hence,
\begin{equation*}
\label{eq:conv2} \widehat{T^{2\varphi(k)}F}\xrightarrow[k\to\infty]{}\widehat{F_\infty}.
\end{equation*}
Even better, the uniform integrability of variations stated in Lemma
\ref{lemunifint} makes the above compact convergence perfectly equivalent
to uniform convergence on all $\mathbb R$. In particular,
\begin{equation}
\label{comb1}
\inf_{\mathbb R} \widehat{F_\infty}=\lim_{k\to\infty}\uparrow\inf_{\mathbb R}
\widehat{T^{2\varphi(k)}F}=\gamma^- ~\mbox{and}~ \sup_{\mathbb R}
\widehat{F_\infty}=\lim_{k\to\infty}\downarrow\sup_{\mathbb
  R}\widehat{T^{2\varphi(k)}F}=\gamma^+.
\end{equation} 
On the other hand, a
straightforward use of the the dominated convergence Theorem shows that the
restriction of $T$ to ${\cal D}_M$ is continuous with respect to compact
convergence. Therefore, (\ref{eq:conv}) implies
\begin{equation*}
T^{2(\varphi(k)+1)}F\xrightarrow[k\to\infty]{}T^2F_\infty.
\end{equation*}
But using exactly the same arguments as above (note that $\gamma^-,\gamma^+$ do not depend on $\varphi$), we obtain a
similar conclusion :
\begin{equation}
\label{comb2}
\inf_{\mathbb R}\widehat{T^2F_\infty} = \gamma^- \textrm{ and }
\sup_{\mathbb R}\widehat{T^2F_\infty} =\gamma^+.
\end{equation}
By the second part of Lemma \ref{lm:maxamplitude}, having both (\ref{comb1}) and
(\ref{comb2}) implies that $\gamma^-=\gamma^+$.
\end{proof}

\subsection{Strong attractiveness.}
So far, we have established the distributional convergence of the message
process. To complete the algorithm analysis, we now need to prove
sample-path wise convergence.  We note that Aldous and Bandyopadhyay
\cite{maxtype,bivar} have studied the special case where the i.i.d. initial
messages $(\langle v\to\dot{v}\rangle^0_{\cal T})_{v\in{\cal V}^*}$ are
distributed according to the fixed point $F^*$. They established
$L^2$-convergence of the message process to some unique stationary
configuration which is independent of $(\langle v\to\dot{v}\rangle^0_{\cal
  T})_{v\in{\cal V}^*}$. They call this the {\em bivariate uniqueness
  property}. This sub-section is dedicated to extending such a property
to the case of $F$-distributed i.i.d. initial messages, where $F$ is any tail
distribution satisfying the assumption of Theorem \ref{tm:weakconv},
namely:
\begin{equation}
\label{assumption}\int_{0}^{\infty}{F}<\infty,\textrm{ or equivalently }\mathbb
E\left[\left(\langle v\to\dot{v}\rangle^0_{\cal
T}\right)^+\right]<\infty.
\end{equation}
Recall that, if (\ref{assumption}) does not hold, then $(T^kF)_{k\geq 1}$
simply alternates between the $0$ and $1$ functions. In other words, all
messages in $\cal T$ become almost surely infinite after the very first
iteration. Henceforth, we will assume (\ref{assumption}) to hold, which is
in particular the case if all initial messages are set to zero.  We first
state a Lemma that will allow us to fix the problem of non-continuity of
the update and decision rules on $\cal T$ caused by the infinite number of
terms involved in the minimization.
\begin{lemma}
\label{lm:welldefiniteness}
Under assumption (\ref{assumption}), $\displaystyle{\pi^k_{\cal
    T}(v)=\argmin_{w\sim v}{\left\{\dift{w}{v}{k}\right\}}}$ is a.s. well
defined for every $k\geq 4,v\in\cal V$ despite the infinite number of terms
involved in the argmin. Moreover,
\begin{equation}
\label{eq:unifdomk}
\sup_{k\geq 4}\mathbb P\left(\argmin_{i\geq 1}{\left\{\dift{v.i}{v}{k}\right\}}\geq i_0\right)\xrightarrow[i_0\to\infty]{}0.
\end{equation}
\end{lemma}
With this uniform control in hand, we are now ready to prove the
strong convergence of BP on $\cal T$.
\begin{theorem}[Convergence of BP on $\cal T$]
\label{tm:strongconv} Assume the i.i.d. initial messages satisfy
(\ref{assumption}). Then, up to some additive constant $\gamma\in\mathbb R$, the recursive tree process 
defined by (\ref{eq:recas}) converges to the unique stationary configuration $\langle
\cdot\!\to\!\cdot\rangle^*_{\cal T}$ in the following sense: for every $v\in{\cal V}^*$,
\begin{equation*}
\langle v\to \dot{v}\rangle^{2k}_{{\cal T}} \xrightarrow[k\to\infty]{L^2}
\langle v\to \dot{v}\rangle^*_{{\cal T}}+\gamma \qquad\textrm{ and }\qquad \langle v\to \dot{v}\rangle^{2k+1}_{{\cal T}} \xrightarrow[k\to\infty]{L^2}
\langle v\to \dot{v}\rangle^*_{{\cal T}}-\gamma.\\
\end{equation*}
Further, defining $\pi^*_{\cal T}$ as the assignment induced by $\langle
\cdot\!\to\!\cdot\rangle^*_{\cal T}$ according to rule (\ref{beliefalgo}),
we have convergence of decisions at the root:
\begin{equation*}
\pi^k_{{\cal
      T}}(\varnothing)\xrightarrow[k\to\infty]{P}\pi^*_{{\cal
      T}}(\varnothing).
\end{equation*}
\end{theorem}

\begin{proof}
  Denote by $F$ the tail distribution function of the initial messages. The
  idea is to construct an appropriate stochastic coupling between our
  $F-$initialized message process and the $F^*$-initialized version and
  then use the endogeneity of the latter to conclude.  We let $\gamma$ be
  the constant appearing in Theorem \ref{tm:weakconv}.  First, observe that
  the dynamics (\ref{eq:recas}) are ``anti-homogeneous'': if we add the
  same constant to every initial message, then that constant is simply
  added to every even message $\langle v\to\dot{v}\rangle^{2k}_{\cal T}$
  and subtracted from every odd message $\langle
  v\to\dot{v}\rangle^{2k+1}_{\cal T}$. Therefore, without loss of
  generality we may assume $\gamma=0$. That is, for any $\beps > 0$ there
  exists $k_\varepsilon\in\mathbb N$ so that
$$\theta_{-\varepsilon}F^*\leq T^{k_\varepsilon}F\leq \theta_{\varepsilon}F^*.$$
By a classical result often termed as Strassen's Theorem, probability
measures satisfying such a stochastic ordering can always be coupled in a
pointwise monotone manner. Specifically, there exists a
probability space $E'=(\Omega',{\cal F}',P')$, possibly differing from the
original space $E=(\Omega,{\cal F},P)$, on which can be 
defined a random variable $X^\varepsilon$ with distribution
$T^{k_\varepsilon}F$ and two random variables $X^{-}$
and $X^{+}$ with distribution $F^*$, in such a way that almost surely,
\begin{equation}
X^--\varepsilon \leq X^\varepsilon \leq X^++\varepsilon . 
\label{ineq}
\end{equation}
Now consider the product space $(\bigotimes_{v\in {\cal V}} E')\otimes E$
over which we can jointly define the PWIT ${\cal T}$ and independent copies
$(X^{-}_v,X^{\varepsilon}_v,X_v^{+})_{v\in \cal V}$ of the triple
$(X^-,X,X^+)$ for each vertex $v \in {\cal V}$.  On $\cal T$, let us
compare the configurations $\big(\langle \cdot\!\to\!\cdot\rangle_{{\cal
    T}}^{k,-}\big)_{k\geq 0}$, $\big(\langle
\cdot\!\to\!\cdot\rangle_{{\cal T}}^{k,\varepsilon}\big)_{k\geq 0}$ and
$\big(\langle \cdot\!\to\! \cdot\rangle_{{\cal T}}^{k,+}\big)_{k\geq 0}$
resulting from three different initial conditions, namely:
\begin{equation*}
\forall v \in {\cal V}^*, \left\{
\begin{array}{lll}
\displaystyle{\langle v\to \dot{v}\rangle^{0,-}_{{\cal T}}}& = & X^{-}_v;\smallskip\\
\displaystyle{\langle v\to \dot{v}\rangle^{0,\varepsilon}_{{\cal T}}}& = & X^{\varepsilon}_v ;\smallskip\\
\displaystyle{\langle v\to \dot{v}\rangle^{0,+}_{{\cal T}}}& = &
X^{+}_v.
\end{array}
\right.
\end{equation*}

Due to anti-monotony and anti-homogeneity of the update rule
(\ref{eq:recas}), inequality (\ref{ineq}) `propagates' in the sense that
for any $k\geq 0$ and $v \in {\cal V}^*$,
\begin{equation*}
\label{coupling}
 \begin{array}{lllll}
\langle v\to \dot{v}\rangle_{{\cal T}}^{2k,-}- \varepsilon & 
\leq &\langle v\to \dot{v}\rangle_{{\cal T}}^{2k,\varepsilon} & 
\leq &\langle v\to \dot{v}\rangle_{{\cal T}}^{2k,+} +\varepsilon ;\\
\langle v\to\dot{v}\rangle_{{\cal T}}^{2k+1,+} -\varepsilon & 
\leq &\langle v\to \dot{v}\rangle_{{\cal T}}^{2k+1,\varepsilon} &
\leq &\langle v\to \dot{v}\rangle_{{\cal T}}^{2k+1,-} +\varepsilon.
\end{array}
\end{equation*}
Now fix $v\in{\cal V}^*$. By construction, 
$$\left(\langle v\!\to\! \dot{v}\rangle_{{\cal
      T}}^{k+k_\beps}\right)_{k\geq 0}\stackrel{{\cal D}}{=}
\left(\langle v\!\to\! \dot{v}\rangle_{{\cal T}}^{k,\varepsilon}\right)_{k\geq 0}.$$
In particular, for every $k\geq k_\varepsilon$ we have
\begin{eqnarray*}
\sup_{s,t\geq k}\left\|\langle v\to \dot{v}\rangle_{{\cal
T}}^{s} - \langle v\to \dot{v}\rangle_{{\cal T}}^{t} \right\|_{L^2}
&  = & \sup_{s,t\geq k-k_\varepsilon}\left\|\langle v\to
\dot{v}\rangle_{{\cal T}}^{s,\varepsilon} -\langle v\to
\dot{v}\rangle_{{\cal T}}^{t,\varepsilon} \right\|_{L^2}\\
& \leq & 2\sup_{t\geq k-k_\varepsilon}\left\|\langle v\to
\dot{v}\rangle_{{\cal T}}^{t,\pm}-\langle v\to \dot{v}\rangle_{{\cal
T}}^{*}\right\|_{L^2}+2\varepsilon.\\
\end{eqnarray*}
But from the bivariate uniqueness
  property established by Aldous and Bandyopadhyay
\cite{maxtype,bivar} for the logistic distribution, it follows that
$$\sup_{t\geq k-k_\varepsilon}\!\!\left\|\langle v\!\to\!
  \dot{v}\rangle_{{\cal T}}^{t,\pm}\!\!-\!\langle v\!\to\!
  \dot{v}\rangle_{{\cal T}}^{*}\right\|_{L^2}\xrightarrow[k\to\infty]{}0.$$
Thus, the sequence $\left(\langle v\!\to\! \dot{v}\rangle_{{\cal
      T}}^{k}\right)_{k\geq 0}$ is Cauchy in $L^2$, hence convergent. Using
Lemma \ref{lm:welldefiniteness} to justify the interchange between limit
and minimization, it is not hard to check that the limiting configuration
has to be stationary, i.e. is a fixed point for the recursion
(\ref{eq:recas}), and that the estimates $\pi^k_{\cal T},k\geq 0$ do in
turn converge (in probability) to the estimate $\pi^*_{\cal T}$ associated
with the limiting configuration.  Note that endogeneity implies uniqueness
of the stationary configuration, and therefore $\pi^*_{\cal T}$ is nothing
but the infinite optimal assignment studied in \cite{zeta}.
\end{proof}

\section{Third step: putting things together.}\label{sec:step3}
Finally, we are now in position to complete the proof of Theorem
\ref{tm:unifL1conv}, using the following remarkable result by Aldous.
\begin{theorem}[Aldous, \cite{zeta}]
\label{tm:aldous2}
Let $\pi^*_{\cal T}$ be the assignment associated with the unique
stationary configuration $\langle \cdot\!\to\!\cdot\rangle^*_{\cal T}$.
Then $\pi^*_{\cal T}$ is a perfect matching on ${\cal T}$, and
\begin{equation}
\label{aldouscv}
\left(\Knn, \pi_{\Knn}^{*}\right) \xrightarrow[n\to\infty]{{\cal D}} \left({\cal T},
  \pi^*_{\cal T}\right).
\end{equation}
\end{theorem}
\begin{proofof}{Theorem \ref{tm:unifL1conv}}
Using Theorem \ref{tm:continuity} and Skorokhod's representation Theorem,
the above convergence (\ref{aldouscv}) can be extended to include BP's answer at any fixed step $k$:
\begin{equation*}
\left({\cal K}_{nn},\pi^k_{{\cal
K}_{nn}},\pi^*_{{\cal
K}_{nn}}\right)\xrightarrow[n\to\infty]{{\cal D}}\left({\cal
T},\pi^k_{\cal
T},\pi^*_{{\cal T}}\right).
\end{equation*}
In particular, the probability of getting a wrong decision at the root of
$\Knn$ converges as $n\to\infty$ to the probability of getting a wrong
decision at the root of ${\cal T}$: for all $k \geq 0$,
\begin{eqnarray*}
 \mathbb P\left(\pi^k_{{\cal
K}_{nn}}(\varnothing)\neq \pi^*_{{\cal
K}_{nn}}(\varnothing)\right)\xrightarrow[n\to\infty]{}\mathbb
P\left(\pi^k_{{\cal T}}(\varnothing)\neq \pi^*_{{\cal
T}}(\varnothing)\right).
\end{eqnarray*}
Finally, the symmetry of $\Knn$ lets us rewrite the left-hand side as the
expected fraction of errors $\mathbb E\left[d(\pi^k_{{\cal
      K}_{nn}},\pi^*_{{\cal K}_{nn}})\right]$, and Theorem
\ref{tm:strongconv} ensures that the right-hand side vanishes as $k\to\infty$.
\end{proofof}
\section{Conclusion.}\label{sec:conc}
\noindent In this paper we have established that the BP algorithm finds an
almost optimal solution to the random $n\times n$ assignment problem in
time $O(n^2)$ with high probability. The natural lower bound of
$\Omega(n^2)$ makes BP an (order) optimal algorithm. This result
significantly improves over both the worst-case upper bound for exact
convergence of the BP algorithm proved by Bayati, Shah and Sharma
\cite{BSS} and the best-known computational time achieved by Edmonds and
Karp's algorithm \cite{edmund-karp}. Beyond the obvious practical interest
of such an extremely efficient distributed algorithm for locally solving
huge instances of the optimal assignment problem, we hope that the method
used here -- essentially replacing the asymptotic analysis of the algorithm
as the size of the underlying graph tends to infinity by its exact study on
the infinite limiting structure revealed via local weak convergence -- will
become a powerful tool in the fascinating quest for a general mathematical
understanding of loopy BP.

\bibliographystyle{amsplain}
\bibliography{MOR}
\appendix
\section{Proof of Lemma \ref{lm:unifctrl}.}
\label{app1}
The proof of Lemma \ref{lm:unifctrl} lays upon two technical lemmas stated
below. Essentially, the picture is the following: when $i$ gets large, the
length $\ewn{v}{v.i}$ of the $i^{th}$ shortest edge attached to $v$ in
$\Knn$ becomes large too (Lemma \ref{lm:domweights}), whereas the message
$\mgn{v.i}{v}{k}$ passing along that edge remains reasonably small (Lemma
\ref{lm:unifdom}). Therefore, the resulting contribution $\difn{v.i}{v}{k}$
is too large to matter in the minimization.  In what follows, $|v|$ will
denote the number of letters of the word $v \in {\cal V}$, and
$v_1,\ldots,v_{|v|}$ its consecutive letters (e.g. if $v = 1.2.1.3$ then
$|v| = 4$ and $v_1 = 1, v_2= 2, v_3 =1, v_4 = 3$). Also, we will write
$v_{\leq h}$ for the prefix $v_1\cdots v_h$.

\begin{lemma}[Uniform control on edge weights]\label{lm:domweights}
There exist constants $(M_h)_{h\geq 1}$, $\alpha$ and $\beta>0$ such
that for all $v\in {\cal V}, i\geq 1, t\in\mathbb R^+$, and all $n$ large
enough for ${\Knn}$ to contain $v.i$, 
\begin{eqnarray*}
\label{eq:dom}\mathbb P\Big(\ewn{v}{v.i}\leq
t\Big) \leq M_{|v|}\frac{(\alpha t)^i}{i!}e^{\alpha t} & \textrm{
and }& \mathbb P\Big(\ewn{v}{v.1}\geq t\Big)
\leq M_{|v|}e^{-\beta t}.
\end{eqnarray*}
\end{lemma}
\begin{proof}
  Suppose $\ewn{v}{v.i}\leq t$. Then by construction, the sequence of words
  $(v_{\leq 0},\dots,v_{\leq |v|})$ represents a path in $\Knn$ starting
  from the root and ending at a vertex from which at least $i$ incident
  edges have length at most $t$. Following down this path and deleting
  every cycle we meet, we obtain a cycle-free path $x=(x_0,\ldots,x_k)$
  ($0\leq k\leq |v|\wedge 2n-1$) starting from the root and satisfying
  \begin{equation}
\label{eq:event1} \textrm{card}\Big\{y\sim x_{k},y\neq x_{k-1},
\ewn{x_k}{y}\leq t\Big\}\geq i-1.
\end{equation}
For $0\leq j<k$, $(x_j,x_{j+1})$ corresponds to some $(v_{\leq p-1},v_{\leq
  p})$, $1\leq p\leq |v|$. By definition of our relabeling, the number of
edges in $\Knn$ that are incident to $v_{\leq p-1}$ and shorter than
$\{v_{\leq p-1},v_{\leq p}\}$ is precisely $v_p-1$ or $v_p$, depending on
the parent-edge. Therefore, there exists $p \in \{1,\ldots,|v|\}$ such that
\begin{equation}
\label{eq:event2} 
v_p-\left\lceil\frac k 2\right\rceil\leq
\textrm{card}\Big\{y\notin\{x_{1},\ldots,x_k\},
\ewn{x_j}{y} < \ewn{x_j}{x_{j+1}}\Big\}\leq v_p.
\end{equation}
The $\left\lceil\frac k 2\right\rceil$ above comes from the fact that only
half of the $x_{1},\ldots,x_k$ are neighbors of $x_j$ in $\Knn$.  We thus
have shown that
\begin{equation*}
  \mathbb P\Big(\ewn{v}{v.i}\leq t\Big) \leq
  \sum_{k=0}^{|v|}{\sum_{x=(x_0,...x_k)}{\mathbb P\left(A_{n,x}\cap
        \bigcap_{j=0}^{k-1}B^j_{n,x}\right)}},
\end{equation*}
where the event $A_{n,x}$ corresponds to (\ref{eq:event1}) and the event
$B^j_{n,x}$ to (\ref{eq:event2}). The summation in the above inequality is
over all possible cycle-free paths $x=(x_0,...x_k)$ starting from the root
in ${\Knn}$.  Now since all the edges involved are pairwise distinct, the
events $A_{n,x},B^{0}_{n,x},...,B^{k-1}_{n,x}$ are independent. Moreover,
\begin{eqnarray*}
\mathbb P\left(B^j_{n,x}\right)& = &
\sum_{p=1}^{|v|}{\sum_{q=v_p-\left\lceil\frac k
2\right\rceil}^{v_p}\frac{1}{n+1-\left\lceil\frac k
2\right\rceil}}\leq \frac {(|v|+1)^3} n ;\\
\mathbb P\left(A_{n,x}\right)& = & \mathbb
\displaystyle{\sum_{q=i-1}^{n-1} {n-1\choose q}{H\left(\frac
t{nH'(0)}\right)}^q{\left(1-H\left(\frac
t{nH'(0)}\right)\right)}^{n-1-q} \leq \frac{(\alpha
t)^i}{i!}e^{\alpha t},}
\end{eqnarray*}
where we have used assumption A1 to define ${\alpha=\frac 1
  {H'(0)}\sup_{\varrho\in\mathbb
    R^+}{\frac{H(\varrho)}{\varrho}}}<+\infty.$ This yields the first bound
since there are less than $n^k$ cycle-free paths $x=(x_0,...,x_k)$ starting
from the root in ${\Knn}$. For the second one, the event $A_{n,x}$ is
simply replaced by $\textrm{card}\Big\{y\sim x_k, y\neq x_{k-1},
\ewn{x_k}{y}\leq t\Big\}\leq 1$, whose probability is straightforwardly
exponentially bounded using assumption A2.
\end{proof}
\begin{lemma}[Uniform control on messages]
\label{lm:unifdom} There exist constants
$(M_{k,h},\beta_{k,h})_{k,h\geq 0}>0$ such that for all $v \in {\cal V}^*$
and $t \in {\mathbb R}^+$, uniformly in $n$ (as long as $n$ is large enough so that $v \in \Knn$),
\begin{equation}
\label{eq:unifdom2}
\mathbb P\Big(\left|\mgn{v}{\dot{v}}{k}\right| \geq t\Big)\leq  M_{k,|v|}e^{-\beta_{k,|v|}t}.
\end{equation}
\end{lemma}
\begin{proof}
  The proof is by induction over $k$. The base case of $k = 0$ follows
  trivially. Now, assume (\ref{eq:unifdom2}) is true for a given $k\in
  \mathbb N$. By Lemma \ref{lm:domweights} we can write for all $v\in {\cal
    V}^*$ and $t\in\mathbb R^+$:
\begin{eqnarray*}
  \mathbb P\Big(\mgn{v}{\dot{v}}{k+1}\geq
  t\Big) & = & \mathbb P\Big(\min_{1\leq
    i<n}\big\{\difn{v}{v.i}{k}\big\}\geq t\Big) \\
  & \leq & \mathbb P\Big(\ewn{v}{v.1}\geq
  \frac{t}{2}\Big)+\mathbb P\Big(\mgn{v.1}{v}{k}\leq -\frac t 2\Big)\\
  & \leq &
  M_{|v|}e^{-\frac{\beta}2t}+M_{k,|v|+1}e^{-\frac{\beta_{k,|v|+1}}2t}.
\end{eqnarray*}
The other side is slightly harder to obtain. Again by Lemma
\ref{lm:domweights} :
\begin{eqnarray*}
\mathbb P\Big(\mgn{v}{\dot{v}}{k+1}\leq
-t\Big) & = & \mathbb P\Big(\min_{1\leq
i<n}\big\{\difn{v}{v.i}{k}\big\}\leq -t\Big) \\
& \leq & \sum_{i=1}^{n-1} \mathbb P\Big(\ewn{v}{v.i}\leq r_i(t)\Big)+\sum_{i=1}^{n-1}\mathbb P\Big(\mgn{v.i}{v}{k}\geq t +r_i(t)\Big),\\
& \leq & M_{|v|}\sum_{i=1}^{\infty}\frac{\big(\alpha
r_i(t)\big)^ie^{\alpha
r_i(t)}}{i!}+M_{k,|v|+1}\sum_{i=1}^{\infty}e^{-\beta_{k,|v|+1}(t+r_i(t))},
\end{eqnarray*}
where the inequalities hold for any choice of the quantities $r_i(t)\geq
0$. Our proof thus boils down to the following simple question: can we 
choose the $r_i(t)$ such that
\begin{enumerate}
\item {$r_i(t)$ is large enough to ensure exponential vanishing
of $\displaystyle{f(t)=\sum_{i=1}^{\infty}e^{-\beta_{k,|v|+1}(t+r_i(t))}}$;}
\item {$r_i(t)$ is small enough to ensure exponential vanishing
of $\displaystyle{g(t)=\sum_{i=1}^{\infty}\frac{\big(\alpha
r_i(t)\big)^ie^{\alpha r_i(t)}}{i!}}$.}
\end{enumerate}
The answer is yes. Indeed, taking $r_i(t)=\delta ie^{-\gamma t}$
with $\gamma,\delta >0$ yields 
\begin{eqnarray*}
\frac 1 t \log f(t)\xrightarrow[t\to+\infty]{}\gamma-\beta_{k,|v|+1}
\textrm{ and }\frac 1 t \log g(t) \leq-\gamma + \frac 1 t \log \sum_{i=1}^{\infty}\frac{\big(\alpha\delta e^{\alpha\delta}i\big)^i}{i!}.
\end{eqnarray*}
Therefore, choosing any $\gamma < \beta_{k,|v|+1}$ is enough to ensure (i),
and taking $\delta$ small enough for $\alpha\delta e^{\alpha\delta-1} <
1$ will guarantee (ii) since the right-hand summand is equivalent to $\frac{(\alpha\delta  e^{\alpha\delta-1})^i}{\sqrt{2\pi i}}$  by Stirling's
formula.
\end{proof}
We now know enough to prove Lemma \ref{lm:unifctrl}.
\begin{proofof}{Lemma \ref{lm:unifctrl}}
Set $\delta>0$ small enough to ensure $\alpha\delta
e^{\alpha\delta-1} < 1$. Then, for $t\in\mathbb R^+$,
\begin{eqnarray*}
\lefteqn{\mathbb P\left(\argmin_{1\leq i<
n}\Big\{\difn{v}{v.i}{k} \Big\} \geq
 i_0\right)}\\
& \leq & \mathbb P\left(\mgn{v}{\dot{v}}{k+1}\geq t\right)+ \sum_{i = i_0}^{n-1}\mathbb
P\left(\ewn{v}{v.i}\leq
  \delta  i\right)+\sum_{i =  i_0}^{n-1}\mathbb P\left(\mgn{v.i}{v}{k}\geq \delta i -t\right)\\
& \leq &M_{k+1,|v|}e^{-\beta_{k+1,|v|}t}+
  M_{|v|}\sum_{i =  i_0}^{\infty}{\frac{(\alpha\delta e^{\alpha\delta}i)^i}{i!}}+M_{k,|v|+1}\sum_{i =  i_0}^{\infty}e^{-\beta_{k,|v|+1}(\delta
i-t)},
\end{eqnarray*}
by Lemmas \ref{lm:domweights} and \ref{lm:unifdom}. Letting $i_0\to\infty$ and finally $t\to\infty$ yields the desired result.
\end{proofof}

\section{Proof of the Lemmas in Section \ref{sec:step2}.}
\label{app2} Here we prove Lemmas
\ref{lm:boundedness}, \ref{lm:maxamplitude}, \ref{lemunifint} and
\ref{lm:welldefiniteness}.
\begin{proofof}{Lemma \ref{lm:boundedness}}
From $\int_{0}^{+\infty}F<\infty$ and the definition of $TF\colon x\mapsto
e^{-\int_{-x}^{+\infty}F}$, it follows that:
\begin{enumerate}
\item as $x\to+\infty, TF(x)=\Theta\left(e^{-\int_{-x}^{x_0}F}\right)$ for
  any fixed $x_0\in\mathbb R$;
\item as $x\to-\infty, TF(x)=1-\Theta\left(\int_{-x}^{+\infty}F\right)$.
\end{enumerate}
Now since $F$ is non-zero and non-increasing, there exists $\alpha,\beta>0$
(simply take $\beta = 1$) such that, for all
small enough $x\in\mathbb R$, $\alpha\leq F(x)\leq \beta.$ 
Replacing these inequalities into (i) above yields :
\begin{equation}
\textrm{as }x\to+\infty, TF(x)=O(e^{-\alpha x}) \textrm{ and
}TF(x)=\Omega(e^{-\beta x}).
\end{equation}
In particular, $TF$ satisfies the assumptions made on $F$, so
by induction $T^kF, k\geq 2$ also do, and we may therefore
iteratively apply (i)/(ii) to $TF$, $T^2F$ and $T^3F$. This
successively yields:
\begin{eqnarray}
\textrm{as }x\to-\infty, &&T^2F(x)=1-O(e^{\alpha x})\textrm{ and }T^2F(x)=1-\Omega(e^{\beta x});\\
\label{bigtheta}
\textrm{as }x\to+\infty, &&T^3F(x)=\Theta(e^{-x});\\
\textrm{as }x\to-\infty, &&T^4F(x)=1-\Theta(e^{x}).
\end{eqnarray}
Replacing $F$ by $TF$, we see that (\ref{bigtheta}) also holds for $T^4F$, so we end up with $T^4F(x)$ being both $\Theta(e^{-x})$
as $x\to+\infty$ and $1-\Theta(e^{x})$ as $x\to-\infty$. Besides, on any
compact set, $T^4F$ takes values within a compact subset
of $]0,1[$ by monotonicity. Hence the boundedness of
${\widehat{T^4F}\colon x\mapsto x+\ln\left(\frac{T^4F(x)}{1-T^4F(x)}\right)}$ over $\mathbb R$.
\end{proofof}

\begin{proofof}{Lemma \ref{lm:maxamplitude}}
It follows from the properties (\ref{decr}) and (\ref{tshift})
of $T$ that for every $m, M \in\mathbb R$,
$$\theta_mF^* \leq F \leq \theta_MF^* \Longrightarrow  \theta_{-M}F^* \leq TF
\leq \theta_{-m} F^*.$$ 
Once rewritten in terms of the $\widehat{\cdot}$ transform, this becomes:
$$m \leq \widehat{F} \leq M \Longrightarrow -M\leq \widehat{TF} \leq -m,$$
and the desired inequalities follow by taking $m=\inf_{\mathbb R}{\widehat{F}}$
and $M=\sup_{\mathbb R}{\widehat{F}}$. Now assume $\widehat F$ is
not constant on $\mathbb R$. The right-continuity (of $F$ and
hence) of $\widehat{F}$ ensures existence of an open interval $(a,b)$ such that
$M'=\sup_{(a,b)}{\widehat F}<\sup_{\mathbb R}{\widehat{F}}=M$.
Then, for $x\geq -a$,
\begin{eqnarray*}
TF(x)  =  \exp{\left(-\int_{-x}^{+\infty}{F}\right)}
      & \geq & \exp{\left(-\int_{-x}^{a}{\theta_{M}F^*}-\int_{a}^{b}{\theta_{M'}F^*}-\int_{b}^{\infty}{\theta_{M}F^*}\right)}\\
      &=& \kappa\times \theta_{-M}F^*(x)\ \textrm{ with }\ \kappa=\exp{\left(\int_{a}^{b}{\left(\theta_MF^*-\theta_{M'}F^*\right)}\right)}>1.
\end{eqnarray*}
Applying $T$ again implies that for every $x\in\mathbb R$,
\begin{equation*}
\label{majj} \left\{
\begin{array}{l}
 x\leq a \Rightarrow T^2F(x)  \leq  \displaystyle{\exp{\left(-\kappa\int_{-x}^{+\infty}{\theta_{-M}F^*}\right)} = \left(\theta_{M}F^*(x)\right)^\kappa} ;\\
 x\geq a \Rightarrow T^2F(x)  \leq  \displaystyle{\exp{\left(-\int_{-x}^{-a}{\theta_{-M}F^*}-\kappa\int_{-a}^{+\infty}{\theta_{-M}F^*}\right)}
=\kappa'\times\theta_{M}F^*(x)},
\end{array}
\right.
\end{equation*}
where $\kappa'=\left(\theta_{M}F^*(a)\right)^{\kappa-1}<1$. Now, simply
observing that both $\left(\theta_{M}F^*(x)\right)^\kappa$ and
$\kappa'\times\theta_{M}F^*(x)$ are strictly less than $\theta_{M}F^*(x)$ is
already enough for claiming that $\widehat{T^2F}(x)<M$ for all $x\in\mathbb
R$. In order to conclude that $\sup_{\mathbb R}{\widehat{T^2F}}<M$, we only
need to check that the inequality remains strict at $\pm \infty$:
\begin{equation*}
\left\{
\begin{array}{l}
 \displaystyle{x\leq a \Rightarrow \widehat{T^2F}(x)\leq x +  \ln{\left(\frac{\left(\theta_{M}F^*(x)\right)^\kappa}{1-\left(\theta_{M}F^*(x)\right)^\kappa}\right)}\xrightarrow[x\to-\infty]{}M-\ln{\kappa}<M;}\\
 \displaystyle{x\geq a \Rightarrow \widehat{T^2F}(x)\leq x +  \ln{\left(\frac{\kappa'\times\theta_{M}F^*(x)}{1-\kappa'\times\theta_{M}F^*(x)}\right)}\xrightarrow[x\to+\infty]{}M+\ln{{\kappa'}}<M.}
\end{array}
\right.
\end{equation*}
The inequality $\inf_{\mathbb R}{\widehat{T^2F}}>\inf_{\mathbb R}{\widehat{F}}$ can be obtained in exactly the same way; we skip the details. 
\end{proofof}
\newpage
\begin{proofof}{Lemma \ref{lemunifint}}Fix $k \geq 2$. From  the inequality $|e^{a}-e^b|\leq|a-b|$ for
  all $a,b\leq 0$, and the fact that $0\leq T^{k-2}F\leq 1$, it follows that  $T^{k-1}F\colon x \mapsto
\exp\left(-\int_x^{\infty}T^{k-2}F\right)$ is Lipschitz
continuous with Lipschitz constant $1$. Therefore, $T^{k}F\colon x
\mapsto \exp\left(-\int_x^{\infty}T^{k-1}F\right)$ is differentiable on $\mathbb R$  and for all $x \in \mathbb R$,
 \begin{equation*}
\label{eq:diff} (T^{k}F)'(x) = -T^{k}F(x)\,T^{k-1}F(-x).
\end{equation*}
Hence, $\widehat{ T^{k}F}\colon x \to x+\ln\left(\frac
{T^kF(x)}{1-T^kF(x)}\right)$ is (continuously) differentiable on $\mathbb R$, and for all $x \in \mathbb R$,
\begin{equation}
\label{varphix}
(\widehat{T^{k}F})'(x) = \frac{1-T^{k}F(x)-T^{k-1}F(-x)}{1-T^{k}F(x)}.
\end{equation}
It now remains to check the uniform integrability of $\{(\widehat T^kF)', k\geq 3\}$. Recall that Lemma \ref{lm:maxamplitude} ensures uniform boundedness of the
family $\{\widehat {T^kF},k\geq 0\}$. In other
words, there exists $M\geq 0$ such that:
$$\forall k\geq 0,\, \theta_{-M}F^*\leq T^kF\leq \theta_MF^*.$$
Plugging it into (\ref{varphix}) immediately
yields the uniform bound $|(T^kF)'(x)|\leq \frac{e^{2M}-1}{1+e^{x+M}}$,
which is  enough for uniform integrability on $(0,+\infty)$. For
$(-\infty,0)$ now, observe that the numerator in (\ref{varphix}) vanishes as $x\to-\infty$ and is a
continuously differentiable function of $x$ as soon as $k\geq 3$, with
derivative $$x\mapsto T^{k-1}F(-x)\big(T^kF(x)-T^{k-2}F(x)\big).$$ Therefore,
for all  $k\geq 3$ and $x\in \mathbb R$, 
\begin{eqnarray*}
\big|1-T^{k}F(x)-T^{k-1}F(-x)\big| & = & \left|\int_{-\infty}^x{T^{k-1}F(-u)\big(T^kF(u)-T^{k-2}F(u)\big)\,du}\right|\\
 & \leq & \int_{-\infty}^x{\theta_MF^*(-u)\big(\theta_MF^*(u)-\theta_{-M}F^*(u)\big)du}.
\end{eqnarray*}
Now, the above integrand is $O(e^{2u})$ as
$u\to-\infty$, so the integral is $O(e^{2x})$ as
$x\to-\infty$, whereas the denominator in
(\ref{varphix}) remains always above $1-\theta_MF^*(x)=\Theta(e^{x})$ as $x
\to -\infty$. Thus the resulting bound on $\sup_{k\geq 3}|(\widehat{T^{k}F})'|(x)$ is
$O(e^x)$ as $x \to -\infty$, which is enough for uniform integrability on $(-\infty,0)$.
\end{proofof}

\begin{proofof}{Lemma \ref{lm:welldefiniteness}}
By the
definition of $\cal T$ and the fact that the $k-$step messages sent to $v$ by all its children are i.i.d., we find that for every $i\geq 2$,
\begin{equation}
\mathbb P\left(\dift{v.i}{v}{k}\leq\dift{v.1}{v}{k}\right)=
\mathbb P\left(\xi_{i-1}\leq X_k-Y_k\right),\\
\end{equation}
where $X_k$ and $Y_k$ 
are i.i.d. with distribution $T^kF$ and $(\xi_i)_{i\geq 1}$ is 
a Poisson point process with rate $1$ independent of $X_k,Y_k$.
Now, observe that
\begin{equation}
\label{eq:infinitesum} \sum_{i=1}^\infty{\mathbb P(\xi_{i}\leq
X_k-Y_k)} = \sum_{i=1}^\infty{\mathbb
E\left[\int_0^{{(X_k-Y_k)}^+}{e^{-x}\frac{x^{i-1}}{(i-1)!}\,dx}\right]}
= \mathbb E\left[{\big(X_k-Y_k)}^+\right]\leq E\left[|X^*|\right]+\sup_{\mathbb R}{|\widehat{T^kF}|},
\end{equation}
where $X^*$ is an $F^*-$distributed random variable. It follows from
the previous sub-section that $\sup_{\mathbb R}|\widehat{T^kF}|<+\infty$
as soon as $k\geq 4$, so we can apply the Borel-Cantelli Lemma to get that
 $$\dift{v.i}{v}{k}\geq\dift{v.1}{v}{k}\textrm{ for all large
enough $i$}$$ with probability one, and hence the argmin is well defined. Even better, the boundedness of $\widehat{T^kF}$ derived in the previous sub-section is in fact
uniform in ${k\geq 4}$, and this is enough for (\ref{eq:unifdomk}) to hold.
\end{proofof}

\end{document}